\documentclass[11pt]{article}
\usepackage[english]{babel}
\usepackage{fullpage}

\usepackage[pdfencoding=unicode, psdextra]{hyperref}
\usepackage{eucal,amsmath, amssymb, amsthm}
\usepackage{textcomp, csquotes, subcaption}
\usepackage{tikz}
\usepackage[capitalise]{cleveref}
\usetikzlibrary{positioning,shapes,fit,arrows}

\usepackage{pgf,pgfplots}
\pgfplotsset{compat=1.7}
\usepackage{outlines}
\usepackage{float}
\usepackage[bottom]{footmisc}

\newtheorem{theorem}{Theorem}[section]
\newtheorem{lemma}[theorem]{Lemma}
\newtheorem{corollary}[theorem]{Corollary}
\newtheorem{definition}{Definition}
\newtheorem{conjecture}{Conjecture}
\newtheorem{question}{Question}

\newtheorem{proposition}[theorem]{Proposition}
\newtheorem{claim}{Claim}

\newtheorem{remark}[theorem]{Remark}
\def\S{\mathbb{S}}
\def\A{\mathcal{A}}

\def\B{\mathcal{B}}
\def\C{\mathcal{C}}

\def\K{\mathcal{K}}
\def\H{\mathcal{H}}
\def\I{\mathcal{I}}
\def\J{\mathcal{J}}
\def\bd{\mathrm{BD}}
\def\d{\mathrm{deg}}

\numberwithin{equation}{section}

\vspace{5cm}

\begin{document}

\begin{center}
{ 
       {\Large \textbf { \sc  Higher matching complexes of complete graphs and complete bipartite graphs    
          }
       }
\\

\medskip

 {\sc Anurag Singh\footnote{The author was partially funded by a grant from the Infosys Foundation.}}\\
{\footnotesize Indian Institute of Technology (IIT) Bhilai, India}\\

{\footnotesize e-mail: {\it anurags@iitbhilai.ac.in}}}\\
\end{center}

\thispagestyle{empty}


\begin{abstract}  
{For $r\geq 1$, the $r$-matching complex of a graph $G$, denoted $M_r(G)$, is a simplicial complex whose faces are the subsets $H \subseteq E(G)$ of the edge set of $G$ such that the degree of any vertex in the induced subgraph $G[H]$ is at most $r$. In this article, we give a closed form formula for the homotopy type of the $(n-2)$-matching complex of complete graph on $n$ vertices. We also prove that the $(n-1)$-matching complex of complete bipartite graph $K_{n,n}$ is homotopy equivalent to a sphere of dimension $(n-1)^2-1$.
}
 \end{abstract}
 \hrulefill




\section{{Introduction}}
A {\itshape matching} in a graph $G$ is a subset $H \subseteq E(G)$ of the edge set of $G$ such that the degree of any vertex in the induced subgraph $G[H]$ is at most $1$. The {\itshape matching complex} of $G$, denoted $M_1(G)$, is a simplicial complex whose vertices are the edges of $G$ and faces are all the matchings in $G$. Matching complexes have emerged from developments in many fields, and have shown to be combinatorial objects with rich and subtle topological structure. In particular, these complexes play a fundamental role in the study of Quillen order complexes associated to posets of (elementary abelian) primary subgroups of a given finite group. Subsequently, and starting from the work of Bouc \cite{Bou92}, the study of matching complexes has become a very active and fruitful area of research on its own. Likewise, Chessboard complexes, {\itshape i.e.}, matching complexes of complete bipartite graphs, arise in Garst's work \cite{Gar79} as coset complexes associated to symmetric groups, and also play a role in Vre\'cica-Zivaljevi\'c's analysis of halving hyperplanes (see \cite{ZV92}).
For more on these complexes, interested reader is referred to Wach's survey paper \cite[Section $2$]{Wach03} .

Reiner and Roberts \cite{RR00} generalized the concept of matching complexes by defining bounded degree complexes. Let $G$ be a graph and $V(G)=\{1,2,\dots,n\}=[n]$ be the vertex set of $G$. For $\vec{\lambda}=(\lambda_1,\dots,\lambda_n)$ a sequence of non-negative integers, the {\itshape bounded degree complex}, denoted $\bd^{\vec{\lambda}}(G)$, is a simplicial complex whose simplicies are all subgraphs $H \subseteq E(G)$ such that the degree of vertex $i$ in $H$ is at most $\lambda_i$ for each $i \in [n]$. Jonsson \cite{Jon08} further studied these complexes and derived a lower bound for the connectivity of $\bd^{\vec{\lambda}}(G)$. When $\lambda_i =r$ for all $i \in [n]$, the bounded degree complex $\text{BD}^{(r,\dots,r)}(G)$ is called the {\itshape $r$-matching complex} of $G$ and is denoted by $M_r(G)$ (see \Cref{fig:example of BD} for example). 

There are only a few classes of graphs for which the exact homotopy type of higher matching complexes are known. For example, for trees \cite{Anu20,singh21}, wheel graphs \cite{Vega19}, and cycle graphs \cite{Koz07, Anu20}, $M_r(G)$ has the homotopy type of wedge of spheres for each $r\geq 1$. However, for complete graphs, the homotopy type of these complexes are a bit more mysterious. For example, Shareshian and Wachs \cite{SW07} showed that the integral homology groups of $1$-matching complexes of complete graphs are not torsion-free in many cases. For $2\leq r <n$, Jonsson \cite[Theorem 12.8]{Jon08} computed a lower bound for the connectivity degree (see \Cref{definition:connectivity}) of $M_r(K_n)$ and observed that the integral homology groups of higher matching complexes of complete graphs can also have torsion (see \cite[Table 12.2]{Jon08}). Due to this, in general, it is extremely difficult to determine the explicit homotopy type of higher matching complexes of complete graphs.

In this paper, we determine the homotopy type of $M_{n-2}(K_n)$. In particular, we prove the following.

\begin{theorem}\label{theorem:completegraphmain}
For $n \geq 3$, the complex $M_{n-2}(K_n)$ is homotopy equivalent to a wedge of spheres. More precisely, $$M_{n-2}(K_n)\simeq \bigvee\limits_{n-1}\S^t,$$ 
where $t={\binom{n-1}{2}-1}$.
\end{theorem}

Another class of graphs whose matching complexes have been studied extensively is complete bipartite graphs $K_{m,n}$. In \cite{Zieg94}, Ziegler showed that $M_1(K_{m,n})$ is shellable whenever $n\geq 2m - 1$. Shareshian and Wachs \cite{SW07} studied the integral homology groups of $M_1(K_{m,n})$ and showed that they are not torsion-free in many cases. In \cite[Theorem 3.3]{RR00}, Reiner and Roberts computed the rational homology groups of higher matching complexes of complete bipartite graphs, also known as {\itshape chessboard complexes with multiplicities}. Combining \cite[Proposition $2.4$]{RR00}, \cite[Proposition $3.2$]{RR00} and \cite[Theorem 3.9]{dong02} one gets that the complex $M_{n-1}(K_{n,n})$ is \emph{stably homotopy equivalent} (see \cite[Section $7.1$]{tom08}) to a sphere of dimension $(n-1)^2-1$. Here, we strengthen this result by showing that these complexes are homotopy equivalent. More precisely, we have the following result. 

\begin{theorem}\label{theorem:completebipgraphmain}
For $n\geq 2$, the $(n-1)$-matching complex of complete bipartite graph $K_{n,n}$ is homotopy equivalent to $((n-1)^2-1)$-dimensional sphere.
\end{theorem}

This article is organized as follows: In Section $2$, we present some definitions and tools which are crucial for this article. Section $3$ is dedicated towards the study of $(n-2)$-matching complex of $K_n$. In Section $4$, we prove \Cref{theorem:completebipgraphmain}. Finally, in Section $5$, we outline some open problems.

 
\section{Preliminaries}
A {\itshape graph} is an ordered pair $G=(V,E)$ where $V$ is called the set of vertices and $E$ is the set of (not necessarily all) cardinality-2 subsets of $V$, called the set of edges of $G$. The vertices $v_1, v_2 \in V$ are said to be adjacent, if $\{v_1,v_2\}\in E$. The number of vertices adjacent to a vertex $v$ in $G$ is called {\itshape degree} of $v$ in $G$, denoted deg$_G(v)$. 
A vertex $v$ is said to be {\itshape adjacent} to an edge $e$, if $v$ is an end point of $e$, {\itshape i.e.}, $e=(v, w)$. 

A graph $H$ with $V(H) \subseteq V(G)$ and $E(H) \subseteq E(G)$ is called a {\it subgraph} of the graph $G$. For a nonempty subset $H$ of $E(G)$, the induced subgraph $G[H]$, is the subgraph of $G$ with $V(G[H]) =V(G) $ and edges $E(G[H]) = H$.

For $n \geq 1$, the {\itshape complete graph}, denoted $K_n$, is a graph with vertex set $V(K_n) = \{1, \ldots, n\}$ and  edge set $E(K_n) = \{\{i,j\} : 1 \leq i<j \leq n\}$. For $m,n \geq 1$, the {\itshape complete bipartite graph}, denoted $K_{m,n}$, is a graph with vertex set $V(K_{m,n}) = \{a_1, \ldots, a_m\}\sqcup \{b_1,\dots,b_n\}$ and  edge set $E(K_{m,n}) = \{\{a_i,b_j\} : 1 \leq i \leq m,~ 1 \leq j \leq n\}$.

\begin{definition}
\normalfont{An {\itshape (abstract) simplicial complex} $\K$ on a finite set $X$ is a collection of subsets of $X$ such that 
\begin{itemize}
    \item[$(i)$] $\emptyset \in \K$, and 
    \item[$(ii)$] if $\sigma \in \K$ and $\tau \subseteq \sigma$, then $\tau \in \K$.
\end{itemize}}
\end{definition}

 The elements  of $\K$ are called {\itshape simplices} of $\K$.  Inclusion-wise maximal simplices of $\K$ are called {\itshape facets} of $\K$. If $\sigma \in \K$ and $|\sigma |=k+1$, then $\sigma$ is said to be {\itshape $k$-dimensional}, denoted as dim$(\sigma)=k$ (here, $|\sigma|$ denotes the cardinality of $\sigma$ as a set). A complex  is called {\itshape pure} if all facets are of same dimension. Further, if $\sigma \in \K$ and $\tau \subseteq \sigma$ then $\tau$ is called a {\itshape face} of $\sigma$ and if $\tau \neq \sigma$ then $\tau$ is called a {\itshape proper face} of $\sigma$. The set of $0$-dimensional simplices of $\K$ is denoted by $V(\K)$, and its elements are called {\itshape vertices} of $\K$. A {\itshape subcomplex} of a simplicial complex $\K$ is a simplicial complex whose simplices are contained in $\K$. For $s\geq 0$, the {\itshape $k$-skeleton} of a simplicial complex $\K$, denoted $\K^{(s)}$, is the collection of all those simplices of $\K$ whose dimension is at most $s$.
 In this article, we do not distinguish between an abstract simplicial complex and its geometric realization. Therefore, a simplicial complex will be considered as a topological space, whenever needed. 
 
 For $j\geq 0$, simplicial complex $\K$ is called {\itshape $j$-connected} if, for all $d \in \{0,\dots , j\}$, every continuous map $f: \S^d \rightarrow \K$ has a continuous extension $g: \mathbb{B}^{d+1}\rightarrow \K$. Here, $\S^d$ and $\mathbb{B}^d$ denote the $d$-dimensional sphere and closed ball respectively.  By convention, $\K$ is $(-1)$-connected if it is nonempty.
 
 \begin{definition}\label{definition:connectivity}
  \normalfont{The {\itshape connectivity degree} of a simplicial complex $\K$ is the largest integer $j$ such that $\K$ is $j$-connected ($+\infty$ if $\K$ is $j$-connected for all $j$).} The {\itshape shifted connectivity degree} of $\K$ is obtained by adding one to the connectivity degree.
 \end{definition}

\begin{definition}
\normalfont For $k\geq 1$, a \emph{$k$-matching} of a graph $G$ is a subset of edges $H\subseteq E(G)$ such that any vertex $v \in G[H]$ has degree at most $k$. The \emph{$k$-matching complex} of a graph $G$, denoted $M_k(G)$, is a simplicial complex whose vertices are the edges of $G$ and faces are given by $k$-matchings of $G$. 
\end{definition}

{\bf Example:}
\Cref{fig:example of BD} consists of graph $G$ and $M_2(G)$. The complex $M_2(G)$ consists of $3$ maximal simplices, namely $\{e_1,e_2,e_4\}, \{e_1,e_3,e_4\}$ and $\{e_2,e_3,e_4\}$.

\begin{figure}[H]
	\begin{subfigure}[]{0.45 \textwidth}
		\centering
		\begin{tikzpicture}
 [scale=0.35, vertices/.style={draw, fill=black, circle, inner sep=1.0pt}]
        \node[vertices, label=below:{$1$}] (v1) at (0,0)  {};
		\node[vertices, label=below:{$2$}] (v2) at (6,0)  {};
		\node[vertices, label=above:{$3$}] (v3) at (-2.5,5)  {};
		\node[vertices, label=above:{$4$}] (v4) at (3.5,5)  {};
		
\foreach \to/\from in {v1/v2}
\path (v1) edge node[pos=0.5,below] {$e_1$} (v2);
\path (v1) edge node[pos=0.5,left] {$e_2$} (v3);
\path (v1) edge node[pos=0.75,right] {$e_3$} (v4);
\path (v2) edge node[pos=0.25,above] {$e_4$} (v3);
\end{tikzpicture}\caption{$G$}
	\end{subfigure}
	\begin{subfigure}[]{0.45 \textwidth}
		\centering
	\begin{tikzpicture}
 [scale=0.35, every node/.style={draw=none}]
\node[label=left:{$e_4$}, draw, fill=black, circle, inner sep=1.0pt] (a) at (0,0) {};
\node[label=left:{$e_2$},draw, fill=black, circle, inner sep=1.0pt] (b) at (-2,4) {};
\node[label=left:{$e_3$},draw, fill=black, circle, inner sep=1.0pt] (c) at (-2,-4) {};
\node[label=right:{$e_1$},draw, fill=black, circle, inner sep=1.0pt] (f) at (4,0) {};
\node[label=left:{$e_4$},draw=none,inner sep=0.0pt] (g) at (-4,-0.5) {};

\foreach \to/\from in {a/b,a/c,a/f,b/c,b/f}
\draw [-] (\to)--(\from);
\filldraw[fill=gray!10, draw=black] (0,0)--(-2,4)--(-2,-4)--cycle;
\filldraw[fill=gray!10, draw=black] (0,0)--(-2,4)--(4,0)--cycle;
\filldraw[fill=gray!10, draw=black] (0,0)--(4,0)--(-2,-4)--cycle;
\draw[dashed,thick, ->] (g) to (a);
\end{tikzpicture}\caption{$M_2(G)$}\label{fig:BD of G}
	\end{subfigure}
	\caption{$2$-matching complex of a graph $G$} \label{fig:example of BD}
\end{figure}
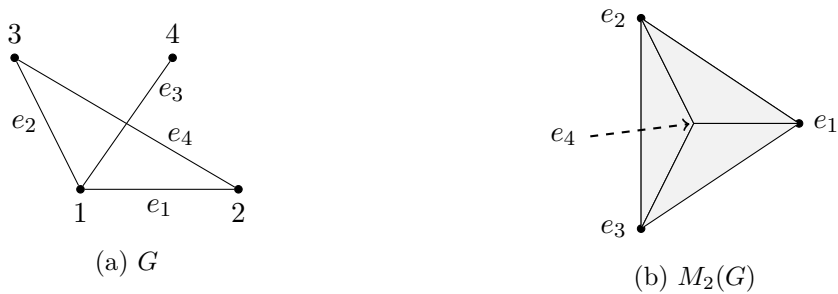

Hereafter, any face $H \in M_r(G)$ will be considered as an induced subgraph $G[H]$, whenever needed.

\subsection{Tools from discrete Morse theory}
We now discuss some tools needed from discrete Morse theory. The classical reference for this is \cite{Forman98}. However, here we closely follow \cite{Koz07} for notations and definitions.

\begin{definition}[{\cite[Definition 11.1]{Koz07}}]
\normalfont A {\itshape partial matching} on a poset $(P,<)$ is a subset $\mathcal{M} \subseteq P \times P$ such that
\begin{itemize}
\item[(i)] $(a,b) \in \mathcal{M}$ implies $ a \prec b;$ {\itshape i.e.}, $a < b$ and no $c$ satisfies $ a < c < b $, and
\item[(ii)] each $ a \in P $ belong to at most one element in $\mathcal{M}$.
\end{itemize}
\end{definition}

   Note that, $\mathcal{M}$ is a  partial matching on a poset $P$ if and only if 
  there exists  $\mathcal{A} \subset P$ and an injective map $\mu: \mathcal{A}
 \rightarrow P\setminus \mathcal{A}$ such that $\mu(a)\succ a$ for all $a \in \mathcal{A}$. 

 	An {\itshape acyclic matching} is a partial matching  $\mathcal{M}$ on the poset $P$ such that there does not exist a cycle
 	\begin{align*}
 		\mu(a_1)  \succ a_1 \prec \mu( a_2) \succ a_2  \prec \mu( a_3) \succ a_3 \dots   \mu(a_t) \succ a_t  \prec \mu(a_1), t\geq 2.
 \end{align*}

For an acyclic partial matching on $P$, those elements of $P$ which do not belong to the matching are called 
{\itshape critical}.

The main result of discrete Morse theory is the following.

\begin{theorem}[{\cite[Theorem 11.13]{Koz07}}]\label{acyc3}
Let $\K$ be a simplicial complex and $\mathcal{M}$ be an acyclic matching on the face poset of $\K$. Let $c_i$ denote the number of critical $i$-dimensional cells of $\K$ with respect to the matching $\mathcal{M}$. Then $\K$ is homotopy equivalent to a cell complex $\K_c$ with $c_i$ cells of dimension $i$ for each $i \geq 0$, plus a single $0$-dimensional cell in the case where the empty set is also paired in the matching.
\end{theorem}

The following can be inferred from Theorem \ref{acyc3}.

\begin{corollary}\label{acyc4}
If an acyclic matching on the face poset of $\K$ has critical cells only in a fixed dimension $i$, then $\K$ is homotopy equivalent to a wedge of $i$-dimensional spheres.
\end{corollary}

In this article, by matching on a simplicial complex $\K$, we will mean that the matching is on the face poset of $\K$. 

\subsection{Morse matching induced by a sequence of vertices}Let $\K$ be a simplicial complex and $N_x= \{\sigma \in \K : \sigma \setminus \{x\},\ \sigma \cup  \{x\} \in \K\}$ be a subcomplex of $\K$, where $x \in V(\K)$. Define a matching on $\K$ using $x$ as follows: 
\begin{equation*}
M_x =\{(\sigma \setminus \{x\} , \ \sigma \cup \{x\}) :  \sigma \setminus\{ x\}, \ \sigma \cup \{x\} \in \K\}.
\end{equation*}

Note that the condition $\sigma\setminus \{x\} \in \K$, for $N_x$ and $M_x$ above, is superfluous since $\K$ is a simplicial complex. However, this is not the case when we define an element matching on a subset of a simplicial complex (for instance, see \Cref{proposition:sequence of element matching}). 
\begin{definition}
\normalfont Matching $M_x$, as defined above, is called an {\itshape element matching} on $\K$ using vertex $x$.
\end{definition}

The following result tells us that an element matching is always acyclic.

\begin{lemma}[{\cite[Lemma 3.2]{NA18}}]\label{lemma:element matching}
The matching $M_x$ is an acyclic matching on $\K$ and  perfect acyclic matching on $N_x$.
\end{lemma}

To obtain an acyclic matching on a simplicial complex $\K$, the next result tells us that one can define a sequence of element matchings on $\K$ using its vertices. 

\begin{proposition}[{\cite[Proposition 3.1]{SSA19}}]\label{proposition:sequence of element matching}
Let $\K_1$ be a simplicial complex and $x_1,x_2,\dots,x_n$ be vertices of $\K_1$. Then, $\bigsqcup\limits_{i=1}^{n} M_{x_i}$ is an acyclic matching on $\K_1$, where $M_{x_i} = \{ (\sigma \setminus \{x_i\} , \ \sigma \cup  \{x_i\}) :  \sigma \setminus\{ x_i\}, \ \sigma \cup \{x_i\} \in \K_i \}$ and $\K_{i+1}=\K_i \setminus \{\sigma : (\sigma\setminus \{x_i\}, \sigma \cup \{x_i\})  \in M_{x_i}\}$ for $i \in \{1,\dots,n\}$.
\end{proposition}

Note that the above result is a particular case of a more general result called Cluster Lemma, which has been re-discovered by many authors over and over (see for instance, \cite[Lemma 4,2]{Jon08} or \cite[Lemma 4.1]{Hersh05}).  \Cref{proposition:sequence of element matching} will be used repeatedly in this article.

\section{Proof of \texorpdfstring{\Cref{theorem:completegraphmain}}{1}}

 The set of vertices of $K_n$ will be denoted by $[n]$ and the notation $G+\{i,j\}$ will mean that add the edge $\{i,j\}$ to $G$ if it is not already there. Similarly, $G-\{i,j\}$ will mean that delete the edge $\{i,j\}$ from $G$ if it is available. 

The aim of this section is to determine the homotopy type of $M_{n-2}(K_n)$ for $n\geq 3$. We do so by defining a sequence of element matchings on $M_{n-2}(K_n)$ in $n-1$ steps.

\vspace*{0.3cm}

\noindent{\bf \large{Step $\mathbf{1}$:}}
Let $\A_{1,1}=M_{n-2}(K_n)$. We first use the vertex $\{1,2\}$ of $\A_{1,1}$ for element matching. Define,
\begin{equation*}
    \begin{split}
        M_{1,2} & = \{(G,G + \{1,2\} ) : \{1,2\} \notin E(G)
        \mathrm{~and~} G, G +\{1,2\} \in \A_{1,1} \}, \\
        N_{1,2} & = \{G \in \A_{1,1} : (G-\{1,2\} ,G+\{1,2\}) \in M_{1,2}\}, \text{ and}\\
        \A_{1,2} & = \A_{1,1} \setminus N_{1,2}.
    \end{split}
\end{equation*}

Using \Cref{lemma:element matching}, we get that $M_{1,2}$ is an acyclic matching on $\A_{1,1}$ with $\A_{1,2}$ as the set of the critical cells.

\begin{claim}\label{claim:step1complete}
$\A_{1,2} =  \{G \in \A_{1,1} : \{1,2\} \notin E(G),~ \d_G(2)<n-2, ~\d_G(1)=n-2\} \sqcup \{G \in \A_{1,1} : \{1,2\} \notin E(G),~ \d_G(2)=n-2\}.$
\end{claim}
\begin{proof}[Proof of \Cref{claim:step1complete}]
For simplicity of notations, let
\begin{equation}\label{eq:B12}
\begin{split}
\B_{1,2} & = \{G \in \A_{1,1} : \{1,2\} \notin E(G),~ \d_G(2)<n-2, ~\d_G(1)=n-2\}, \text{ and}\\
\C_{1,2} & =\{G \in \A_{1,1} : \{1,2\} \notin E(G),~ \d_G(2)=n-2\}.
\end{split}
\end{equation}

Clearly, $\B_{1,2}$ are $\C_{1,2}$ are disjoint sets. Further, if $G \in \B_{1,2} \sqcup \C_{1,2}$, then $\{1,2\} \notin E(G)$ and $\d_G(1)=n-2$ or $\d_G(2)=n-2$. This implies that $G+\{1,2\} \notin \A_{1,1}$, and hence $G \notin N_{1,2}$. Therefore, $\B_{1,2} \sqcup \C_{1,2} \subseteq \A_{1,2}$.

Now consider $G \in \A_{1,2}$. If $\{1,2\}\in E(G)$, then clearly $(G-\{1,2\},G) \in M_{1,2}$ which is a contradiction. Further, if $\{1,2\} \notin E(G)$, $\d_G(1) <n-2$ and $\d_G(2)<n-2$ then also $G+\{1,2\} \in N_{1,2}$. Therefore, if $G \in \A_{1,2}$ then $\{1,2\} \notin E(G)$ and, either $\d_G(2)=n-2$ or $\d_G(1)=n-2$ whenever $\d_G(2)<n-2$.
\end{proof}

We now extend the matching $M_{1,2}$ on $\A_{1,1}$ by defining a sequence of element matchings on $\A_{1,2}$ using vertices $\{1,i\}$ for each $i \in \{3,\dots,n\}$. For $i \in \{3,\dots,n\}$, define
\begin{equation*}\label{eq:matchingstep1}
    \begin{split}
        M_{1,i} & = \{(G,G + \{1,i\} ) : \{1,i\} \notin E(G) \mathrm{~and~} G, G +\{1,i\} \in \A_{1,i-1} \}, \\
        N_{1,i} & = \{G \in \A_{1,i-1} : (G-\{1,i\}, G+\{1,i\})  \in M_{1,i}\}, \text{ and}\\
        \A_{1,i} & = \A_{1,i-1} \setminus N_{1,i}.
    \end{split}
\end{equation*}

Using \Cref{proposition:sequence of element matching}, we get that $\bigsqcup\limits_{2\leq i \leq n}M_{1,i}$ is an acyclic matching on $\A_{1,1}$ with $\A_{1,n}$ as the set of the critical cells. We now analyze cells of $\A_{1,n}$.

\begin{claim}\label{claim:step1complete2}
$\A_{1,n} =  \B_{1,2} \sqcup \{G \in \A_{1,1} : \{1,i\} \notin E(G) \mathrm{~and~} \d_G(i)=n-2 \mathrm{~for~each~} i \in \{2,\dots,n\}\}.$
\end{claim}
\begin{proof}[Proof of \Cref{claim:step1complete2}]
Define
\begin{equation}\label{eq:C1}
\C_1  = \{G \in \A_{1,1} : \{1,i\} \notin E(G) \mathrm{~and~} \d_G(i)=n-2 \mathrm{~for~each~} i \in \{2,\dots,n\}\}.
\end{equation}

Clearly, if $G \in \C_1$ then, for each $i \in \{3,\dots, n\}$, $\d_{G+(1,i)}(i)=n-1$  which implies that $G+\{1,i\} \notin \A_{1,2}$. This gives us that $G \notin N_{1,i}$ for any $i \in \{3,\dots, n\}$, thereby showing that $G \in \A_{1,n}$. If $G \in \B_{1,2}$ then $\d_G(2)<n-2$ and $\d_{G-(1,i)}(1)=n-3$ for each $i \in \{3,\dots,n\}$. This gives us that $G-\{1,i\} \notin \A_{1,2}$ for any $i \in \{3,\dots,n\}$. Therefore, $\B_{1,2}\sqcup \C_1 \subseteq A_{1,n}$.

Now consider $G \in \A_{1,n}\subseteq \A_{1,2}$. If $\d_G(2) < n-2$ then clearly $G \in \B_{1,2}$. Let $\d_G(2)=n-2$ and $G \notin \C_1$, {\itshape i.e.} $G \in \A_{1,2}\setminus \C_1$. Then, either $\{1,i\} \in E(G)$ for some $i \in \{2,\dots,n\}$ or $\d_G(j)<n-2$ for some $j \in \{2,\dots,n\}$.
Let $i_0 = \text{min}\{i : \{1,i\} \in E(G) \text{ or } \d_G(i)<n-2\}$. Since $G \in \A_{1,2}$, $i_0>2$. In this case, it is easy to see that $G \in N_{1,i_0}$ which contradicts the assumption that $G \in \A_{1,n}$. 
\end{proof}

Our idea to define matchings in each coming steps will be similar to this step. In step $i$, we are going to use vertices $\{i,j\}$ for $j\in \{i+1,\dots,n\}$. To make our writing easier in the next step, we observe the following.

\begin{proposition}\label{prop:C1 is critical}
 Let $\B_{1,2}$ and $\C_{1}$ be the sets as defined in \Cref{eq:B12} and \Cref{eq:C1} respectively. Then,
\begin{enumerate}
\item $|\C_1|=1$ and if $G \in \C_1$ then $|E(G)|= \binom{n-1}{2}$.
\item If $G \in \C_1$ then $G-\{i,j\}, G+\{i,j\} \notin \B_{1,2}$ for any $2\leq i < j \leq n$.
\end{enumerate}
\end{proposition}
\begin{proof}
If $G \in \C_1$ then vertex $1$ is isolated and any other two vertices are joined by an edge, {\itshape i.e.} $G$ is disjoint union of an isolated vertex $\{1\}$ and a complete graph on vertex set $\{2,\dots,n\}$. This proves the first part.

Now consider $G \in \C_1$. Since $\{i,j\} \in G$ for each $2\leq i < j \leq n$, $G+\{i,j\}=G \notin \B_{1,2}$. Further, $\d_{G-(i,j)}(1)=0$ for all $2\leq i < j \leq n$, which implies that $G-\{i,j\} \notin \B_{1,2}$ for any $2\leq i < j \leq n$.
\end{proof}

From \Cref{prop:C1 is critical}, it is clear that $\C_1$ is not going to play any role in any element matching using vertices $\{i,j\}$ where $2\leq i <j \leq n$. Therefore, for the time being it is enough to proceed with set $\B_{1,2}$.

\vspace*{0.3cm}

\noindent{\bf \large{Step $\mathbf{2}$:}} Let $\A_{2,2}=\B_{1,2}=\A_{1,n}\setminus \C_1$. We now define a sequence of element matchings on $\A_{2,2}$ using vertices $\{2,i\}$ for each $i \in \{3,\dots,n\}$. For $i \in \{3,\dots,n\}$, define
\begin{equation}\label{eq:matchingstep2}
    \begin{split}
        M_{2,i} & = \{(G,G + \{2,i\} ) : \{2,i\} \notin E(G) \mathrm{~and~} G, G +\{2,i\} \in \A_{2,i-1} \}, \\
        N_{2,i} & = \{G \in \A_{2,i-1} : (G-\{2,i\},G+\{2,i\}) \in M_{2,i}\}, \text{ and}\\
        \A_{2,i} & = \A_{2,i-1} \setminus N_{2,i}.
    \end{split}
\end{equation}

\begin{proposition}\label{prop:step2complete}
For $i \in \{2,\dots,n\}$, let $\A_{2,i}$ be as defined above.
\begin{enumerate}
\item $\A_{2,3} =  \B_{2,3} \sqcup \C_{2,3}$, where
\begin{equation}\label{eq:B23}
\begin{split}
\B_{2,3} & = \{G \in \A_{2,2} : \{2,3\} \notin E(G),~ \d_G(3)<n-2, ~\d_G(2)=n-3\}, \text{ and}\\
\C_{2,3} & =\{G \in \A_{2,2} : \{2,3\} \notin E(G),~ \d_G(3)=n-2\}.
\end{split}
\end{equation}

\item $\A_{2,n} =  \B_{2,3} \sqcup \C_2$, where
\begin{equation}
\C_2=\{G \in \A_{2,2} : \{2,i\} \notin E(G) \mathrm{~and~} \d_G(i)=n-2 \mathrm{~for~each~} i \in \{3,\dots,n\}\}.
\end{equation}

\item $|\C_2|=1$ and if $G \in \C_2$ then $|E(G)|= \binom{n-1}{2}$.

\item $\B_{2,3} = \emptyset$ if and only if $n=3$.

\item If $G \in \C_2$ and $\B_{2,3}\neq \emptyset$, then $G-\{i,j\}, G+\{i,j\} \notin \B_{2,3}$ for any $3\leq i < j \leq n$.
\end{enumerate}
\end{proposition}

\begin{proof}
\begin{enumerate}
\item Clearly, $\B_{2,3}$ are $\C_{2,3}$ are disjoint sets. Further, if $G \in \B_{2,3} \sqcup \C_{2,3}$, then $\{2,3\} \notin E(G)$ and $\d_G(2)=n-3$ or $\d_G(3)=n-2$. This implies that $G+\{2,3\} \notin \A_{2,2}$, and hence $G \notin N_{2,3}$. Therefore, $\B_{2,3} \sqcup \C_{2,3} \subseteq \A_{2,3}$.

Now consider $G \in \A_{2,3}$. If $\{2,3\}\in E(G)$, then clearly $(G-\{2,3\},G) \in M_{2,3}$ which is a contradiction. Further, if $\{2,3\} \notin E(G)$ and $\d_G(2)<n-3$ as well as $\d_G(3)<n-2$ then also $G+\{2,3\} \in N_{2,3}$. Therefore, if $G \in \A_{2,3}$ then $\{2,3\} \notin E(G)$ and either $\d_G(3)=n-2$ or $\d_G(2)=n-3$ whenever $\d_G(3)<n-2$.

\item If $G \in \C_2$ then, for each $i \in \{3,4,\dots, n\}$, $\d_{G+\{2,i\}}(i)=n-1$  which implies that $G+\{2,i\} \notin \A_{2,2}$. This gives us that $G \notin N_{2,i}$ for any $i \in \{3,\dots, n\}$, thereby showing that $G \in \A_{2,n}$. If $G \in \B_{2,3}$ then $\{2,3\}\notin E(G),~ \d_G(3)<n-2$ and $\d_{G-\{2,i\}}(2)<n-3$ for each $i \in \{4,\dots,n\}$ implying that $G-\{2,i\} \notin \A_{2,3}$ for any $i \in \{4,\dots,n\}$. Therefore, $\B_{2,3}\sqcup \C_2 \subseteq A_{2,n}$.

Let $G \in \A_{2,n}\subseteq \A_{2,3}$. If $\d_G(3) < n-2$ then clearly $G \in \B_{2,3}$. Let $\d_G(3)=n-2$ and $G \notin \C_2$, {\itshape i.e.} $G \in \A_{2,3}\setminus \C_2$. Then, either $\{2,i\} \in E(G)$ for some $i \in \{4,\dots,n\}$ or $\d_G(j)<n-2$ for some $j \in \{4,\dots,n\}$.
Let $i_0 = \text{min}\{i : \{2,i\} \in E(G) \text{ or } \d_G(i)<n-2\}$. It is easy to see that $G \in N_{2,i_0}$ which contradicts the assumption that $G \in \A_{2,n}$.

\item If $G \in \C_2$ then it is easy to see that $G$ is disjoint union of an isolated vertex $\{2\}$ and a complete graph on vertex set $\{1,3,4,\dots,n\}$.

\item To prove this, we just need to look at the definition of $\B_{2,3}$ in expanded form. 
\begin{equation}\label{eq:simple b23}
\begin{split}
\B_{2,3} & = \{G \in \A_{2,2} : \{2,3\} \notin E(G),~ \d_G(3)<n-2, ~\d_G(2)=n-3\}\\
 & =\{G \in M_{n-2}(K_n) : \{1,2\}, \{2,3\} \notin E(G),~\d_G(1)=n-2,\\
 & \hspace*{3.5cm}\d_G(2)=n-3,~ \d_G(3)<n-2\}.
\end{split}
\end{equation}

\Cref{eq:simple b23} clearly implies the result.

\item Let $G \in \C_2$ and $\B_{2,3}\neq \emptyset$. Since $\{i,j\} \in G$ for each $3\leq i < j \leq n$, $G+\{i,j\}=G \notin \B_{1,2}$. Further, $\d_{G-\{i,j\}}(2)=0$ for all $3\leq i < j \leq n$, which implies that $G-\{i,j\} \in \B_{2,3}$ for some $3\leq i < j \leq n$ only when $n=3$, which contradicts the fact that $\B_{2,3}\neq \emptyset$.
\end{enumerate}
\end{proof}

We now move to step $k$, where $2<k<n$. Inductively, let 
\begin{equation}\label{eq:bk-1k}
\begin{split}
\B_{k-1,k}= & \big{\{}G \in M_{n-2}(K_n) : \{i,i+1\}\notin E(G) \text{ for any } i \in [k-1],~\d_G(1)=n-2,\\
 & \hspace*{2.5cm}\d_G(i)=n-3 \text{ for each } i \in \{2,\dots,k-1\},~ \d_G(k)<n-2\big{\}}. 
 \end{split}
\end{equation}
Compare \Cref{eq:bk-1k} with \Cref{eq:simple b23} when $k=3$.  

\vspace*{0.3cm}

\noindent{\bf \large{Step $\mathbf{k}$:}} Let $\A_{k,k} = \B_{k-1,k}$. As is step $2$, here also we define a sequence of element matchings on $\A_{k,k}$ using vertices $\{k,i\}$ for each $i \in \{k+1,\dots,n\}$. For $i \in \{k+1,\dots,n\}$, define
\begin{equation}\label{eq:matchingstepk}
    \begin{split}
        M_{k,i} & = \{(G,G + \{k,i\} ) : \{k,i\} \notin E(G) \mathrm{~and~} G, G +\{k,i\} \in \A_{k,i-1} \}, \\
        N_{k,i} & = \{G \in \A_{k,i-1} : (G-\{k,i\},G+\{k,i\})  \in M_{k,i}\}, \text{ and}\\
        \A_{k,i} & = \A_{k,i-1} \setminus N_{k,i}.
    \end{split}
\end{equation}

The following result analyses the set of critical cells after this step.

\begin{proposition}\label{prop:stepkcomplete}
For $i \in \{k,\dots,n\}$, let $\A_{k,i}$ be as defined above.
\begin{enumerate}
\item $\A_{k,k+1} =  \B_{k,k+1} \sqcup \C_{k,k+1}$, where
\begin{equation*}\label{eq:Bkk+1}
\begin{split}
\B_{k,k+1} & = \{G \in \A_{k,k} : \{k,k+1\} \notin E(G),~ \d_G(k+1)<n-2, ~\d_G(k)=n-3\}, \text{ and}\\
\C_{k,k+1} & =\{G \in \A_{k,k} : \{k,k+1\} \notin E(G),~ \d_G(k+1)=n-2\}.
\end{split}
\end{equation*}

\item $\A_{k,n} =  \B_{k,k+1} \sqcup \C_k$, where
\begin{equation*}
\C_k=\{G \in \A_{k,k} : \{k,i\} \notin E(G) \mathrm{~and~} \d_G(i)=n-2 \mathrm{~for~each~} i \in \{k+1,\dots,n\}\}.
\end{equation*}

\item $|\C_k|=1$ and if $G \in \C_k$ then $|E(G)|= \binom{n-1}{2}$.

\item $\B_{k,k+1} = \emptyset$ if and only if $n=k+1$.

\item If $G \in \C_k$ and $\B_{k,k+1}\neq \emptyset$, then $G-\{i,j\},G+\{i,j\} \notin \B_{k,k+1}$ for any $k+1\leq i < j \leq n$.
\end{enumerate}
\end{proposition}

\begin{proof} Proof of parts $(1),(2), (4)$ and $(5)$ is similar as in the proof of \Cref{prop:step2complete}. To prove part $(3)$, let $G \in \C_k$ and $F$ denote the set of edges $\{\{i,i+1\}: i \in [k-1]\} \sqcup \{\{k,j\}: j \in \{k+1,\dots,n\}\}$. By definition of $\C_k$, $E(G) \cap F = \emptyset$. Further, $\deg_G(i)=n-2$ for each $i \in \{1,k+1,\dots,n\}$ and $\deg_G(j)=n-3$ for each $j\in \{2,\dots,k-1\}$ imply that $E(G)=E(K_n)\setminus F$. Therefore, $\C_k$ contains exactly one graph and $|E(G)|= |E(K_n)|- |F|=\binom{n}{2}-(n-1)=\binom{n-1}{2}$.
\end{proof}

\begin{proof}[Proof of \Cref{theorem:completegraphmain}]
Using \Cref{proposition:sequence of element matching}, we get that $ {\displaystyle \bigsqcup\limits_{\substack{1\leq i \leq n-1, \\ i+1 \leq j \leq n}}M_{i,j}}$ is an acyclic matching on $\A_{1,1}=M_{n-2}(K_n)$ with $\C=\bigsqcup\limits_{i \in [n-1]} \C_i$ as the set of the critical cells. From \Cref{prop:C1 is critical}$(1)$, \Cref{prop:step2complete}$(3)$ and \Cref{prop:stepkcomplete}$(3)$, we have $|\C|=n-1$ and each graph in $\C$ has exactly $\binom{n-1}{2}$ edges. Therefore, \Cref{acyc4} implies that $M_{n-2}(K_n)$ is homotopy equivalent to a wedge of $(n-1)$ spheres of dimension $\binom{n-1}{2}-1$. 
\end{proof}

For $2 \leq d \leq n-1$, Jonsson \cite[Theorem $12.8$]{Jon08} obtained a connectivity bound for $M_d(K_n)$ and stated that ``we do not believe that the derived bound is actually equal to the connectivity degree". Here, we compare the findings of \Cref{theorem:completegraphmain} with Jonsson's result and show that the bound given by him is actually sharp for $d=n-2$. 

\begin{theorem}{\cite[Theorem $12.8$]{Jon08}}\label{thm:jonssonconnectivity}
Let $d\geq 2$ and $n\geq d+1$. Write $n = (d+4)k+r,$ where $d+1\leq r\leq 2d+4$. Then $M_d(K_n)$ is $(\big{\lceil}{\nu_n^d}\big{\rceil}-1)$-connected, where 
\begin{equation*}
\begin{split}
\nu_n^d & = \frac{(d^2+3d-1)n}{2(d+4)}-\frac{\epsilon_d(r)}{2}-1, \mathrm{~and}\\
\epsilon_d(r) & = \frac{3r}{d+4}- \begin{cases}
1 & \mathrm{~if~} r=d+1; \\
2 & \mathrm{~if~} d+2\leq r\leq d+3; \\
3 & \mathrm{~if~} d+4\leq r\leq 2d+3; \\
4 & \mathrm{~if~} r=2d+4.
\end{cases}
\end{split}
\end{equation*}
\end{theorem}

It is easy to see that $\nu_n^{n-2}=\binom{n-1}{2}-1$. Therefore, \Cref{theorem:completegraphmain} implies the following.
\begin{remark} The connectivity bound for $M_{n-2}(K_n)$ given in \Cref{thm:jonssonconnectivity} is sharp.
\end{remark}

\section{Higher matching complexes of complete bipartite graphs}
 
It is easy to see that, for $m > r \geq  n \geq 1$, the complex $M_r (K_{m,n} )$ is the join\footnote{The \emph{join} of two simplicial complexes $K_1$ and $K_2$ is a
simplicial complex whose simplices are disjoint union of simplices of $K_1$ and of $K_2$.} of $n$ copies of the $(r - 1)$-skeleton of an $(m - 1)$-dimensional simplex. Therefore, from \cite[Lemma 2.5]{BW95}, $M_r (K_{m,n} )$ is homotopy equivalent to a wedge of spheres whenever $m > r \geq  n$. In this section, we determine the homotopy type of $M_{n-1}(K_{n,n})$.
 We first fix some notations. 
\begin{equation}\label{notation:bip}
\begin{split}
V(K_{m,n}) &= \{a_i: i \in [m]\}\sqcup \{b_j: j \in [n]\}, \text{ and}\\
E(K_{m,n}) &= \{\{a_i,b_j\}: i \in [m],~ j \in [n]\}.
\end{split}
\end{equation}

\begin{proof}[Proof of \Cref{theorem:completebipgraphmain}]
We prove this by defining a sequence of element matchings on $M_{n-1}(K_{n,n})$ in $(n-1)$-steps.

\vspace*{0.3cm}

\noindent{\bf \large{Step $\mathbf{1}$:}}
Let $\H_{1,0}=M_{n-1}(K_{n,n})$. For $1\leq i \leq n $, define
\begin{equation*}
    \begin{split}
        M_{1,i} & = \{(G,G + \{a_1,b_i\} ) : \{a_1,b_i\} \notin E(G) \mathrm{~and~} G, G + \{a_1,b_i\} \in \H_{1,i-1} \}, \\
        N_{1,i} & = \{G \in \H_{1,i-1} : (G-\{a_1,b_i\},G+\{a_1,b_i\})  \in M_{1,i}\}, \text{ and}\\
        \H_{1,i} & = \H_{1,i-1} \setminus N_{1,i}.
    \end{split}
\end{equation*}

\begin{claim}\label{claim:completebipnn}
$\H_{1,n}=\{G \in \H_{1,0} : \{a_1,b_1\}\notin E(G),~ \deg_G(b_1)<n-1 \mathrm{~and~} \deg_G(a_1)=n-1 \}.$
\end{claim}
\begin{proof}[Proof of \Cref{claim:completebipnn}]
Observe that, if  $G \in \H_{1,1}$ then $\{a_1,b_1\}\notin E(G)$ and either $\deg_G(a_1)=n-1$ or $\deg_G(b_1)=n-1$. Define, 
\begin{equation}
    \begin{split}
        \I_{1,1} & =\{G \in \H_{1,1} : \deg_G(b_1) = n-1\}\\
                 & = \{G \in \H_{1,0} : \{a_1,b_1\} \notin E(G),~ \deg_G(b_1)=n-1\} \mathrm{~and} \\
        \J_{1,1} & =\{G \in \H_{1,1} : \deg_G(b_1) < n-1, ~\deg_G(a_1)=n-1\}\\
                 & = \{G \in \H_{1,0} : \{a_1,b_1\} \notin E(G), ~\deg_G(b_1) < n-1, ~\deg_G(a_1)=n-1\}.
    \end{split}
\end{equation}

Clearly $\H_{1,1}=\I_{1,1}\sqcup \J_{1,1}.$ Further if $G \in \J_{1,1}$, then for each $i \in \{2,\dots,n\}$, $\{a_1,b_i\}\in E(G)$ and $G-\{a_1,b_i\} \notin \H_{1,1}$. Therefore $\J_{1,1} \subseteq \H_{1,n}$. We now show that $\H_{1,n} \subseteq \J_{1,1}$. Let $H \in \H_{1,n} \setminus \J_{1,1}$, {\itshape i.e.}, $\deg_{H}(b_1)=n-1$. Let $t=\text{min}\{i \in \{2,\dots,n\}: \text{either } \{a_1,b_i\} \in E(H) \text{ or } \deg_H(b_i)<n-1\}$. Observe that $t$ exists because, if $\{a_1,b_i\}\notin E(H)$ and $\deg_{H}(b_i)=n-1$ for each $i \in [n]$, then there exists $j\in \{2,\dots,n\}$ such that $\deg_H(a_j)>n-1$ which contradicts the fact that $H \in M_{n-1}(K_{n,n})$. Now, it is easy to see that $H \in N_{1,t}$. Therefore $\H_{1,n} = \J_{1,1}$.
\end{proof}

 We now move to step $k$, where $1<k<n$. At step $k-1$ we defined a sequence of elements matchings using vertices $\{a_{k-1},b_1\} < \dots < \{a_{k-1},b_n\}$. Inductively, let the set of critical cells after step $k-1$ be
\begin{equation}\label{eq:Hk-1n}
\H_{k-1,n}=\{G \in \H_{1,0} : \{a_i,b_1\} \notin E(G), ~\deg_G(a_i)=n-1,~ \forall~i \in [k-1],~\deg_G(b_1)< n-k+1 \}.
\end{equation}
Compare \Cref{eq:bk-1k} with \Cref{claim:completebipnn} for $k=2$. 

\vspace*{0.3cm}

\noindent{\bf \large{Step $\mathbf{k}$:}} Let $\H_{k,0} = \H_{k-1,n}$. Define a sequence of elements matchings on $\H_{k,0}$ using vertices $\{a_{k},b_1\} < \dots < \{a_{k},b_n\}$. For $1\leq j \leq n$, define
\begin{equation*}
    \begin{split}
        M_{k,j} & = \{(G,G + \{a_k,b_j\} ) : \{a_k,b_j\} \notin E(G) \mathrm{~and~} G, G +\{a_k,b_j\} \in \H_{k,j-1} \}, \\
        N_{k,j} & = \{G \in \H_{k,j-1} : (G-\{a_k,b_j\},G+\{a_k,b_j\}) \in M_{k,j}\}, \text{ and}\\
        \H_{k,j} & = \H_{k,j-1} \setminus N_{k,j}.
    \end{split}
\end{equation*}

Since $k <n$, using similar arguments as in the proof of \Cref{claim:completebipnn}, we get that 
$$\H_{k,n}=\{G \in \H_{k,0} : \{a_k,b_1\}\notin E(G),~ \deg_G(b_1)<n-k \mathrm{~and~} \deg_G(a_k)=n-1 \}.$$

After step $n-1$, we have that ${\displaystyle \bigsqcup\limits_{\substack{1\leq i \leq n-1, \\ 1 \leq j \leq n}}M_{i,j}}$ is an acyclic matching on $M_{n-1}(K_{n,n})$ and the set of critical cells is
\begin{equation*}
    \begin{split}
    \H_{n-1,n} & =\{G \in \H_{n-2,n} : \{a_{n-1},b_1\}\notin E(G),~ \deg_G(b_1)<1 \mathrm{~and~} \deg_G(a_{n-1})=n-1 \}\\
    &=\{G \in M_{n-1}(K_{n,n}) : \{a_i,b_1\} \notin E(G), ~\deg_G(a_i)=n-1,~ \forall~i \in [n-1],~\deg_G(b_1)=0\}.
    \end{split}
\end{equation*}

It is easy to see that the set $\H_{n-1,n}$ contains exactly one element which is isomorphic to the complete bipartite graph $K_{n-1,n-1}$ and two isolated vertices namely $a_n$ and $b_1$. \Cref{acyc4} thus implies that 
\begin{equation*} M_{n-1}(K_{n,n}) \simeq \S^{(n-1)^2-1}.
\end{equation*}

This completes the proof of \Cref{theorem:completebipgraphmain}.
\end{proof}

\section{Concluding remarks}
In this section, we list a few interesting open problems.

\subsection{Complexes of graphs with bounded domination number}
For a graph $G$, a set $S\subseteq V(G)$ is called a {\itshape dominating set} of $G$, if for each $v\in V(G)\setminus S$ there exists $s\in S$ such that $\{s,v\}\in E(G)$. The {\itshape domination number} of $G$ is defined to be the cardinality of the minimum dominating set, {\itshape i.e.},
$$\mathrm{dom(G)} = \mathrm{min}\{i : \mathrm{~there~exists~a~dominating~set~of~}G \mathrm{~of~cardinality~} i\}.$$

In \cite{JT19}, Gonz{\'a}lez and Hoekstra-Mendoza studied the complexes of graphs on $n$ vertices with domination number at least $\gamma$, denoted as $D_{n,\gamma}.$ When we fix $n$ and vary $\gamma$, we get the following filtration
\begin{equation}\label{eq:filtration}
\emptyset = D_{n,n} \subset D_{n,n-1} \subset D_{n,n-2} \subset \dots \subset D_{n,2} \subset D_{n,1}= \Delta^{\binom{n}{2}-1}
\end{equation} 

It is easy to observe that $D_{n,n-1}$ is disjoint union $\binom{n}{2}$ vertices. Therefore, the ``first" non-trivial cases are $D_{n,2}$ and $D_{n,n-2}$. Gonz{\'a}lez and Hoekstra-Mendoza \cite{JT19} showed that the complex $D_{n,n-2}$ is homotopy equivalent to a wedge of $2$-dimensional spheres. 

Observe that, $G \in M_{n-2}(K_n)$ if and only if dom$(G) \geq 2$, {\itshape i.e.,} $G \in D_{n,2}$. Therefore, \Cref{theorem:completegraphmain} gives a closed form formula for the homotopy type of $D_{n,2}$ and settles one more spot in \Cref{eq:filtration}. This raises the following question.

\begin{question}
For $n > \gamma \geq 1 $, is $D_{n,\gamma}$ homotopy equivalent to a wedge of spheres?
\end{question}

Here, empty wedge represents a contractible space. Looking at the homotopy type of $D_{n,n-2}$ and $D_{n,2}$, one might be tempted to guess that $D_{n,\gamma}$ is homotopy equivalent to a wedge of equi-dimensional spheres. But that is not the case in general, for instance, using SageMath \cite{sage} one can see that 
\begin{equation*}
    \tilde{H}_i(D_{6,3};\mathbb{Z}) \cong \begin{cases}
    \mathbb{Z}^{115}, & \text{ if } i=4;\\
    \mathbb{Z}^{24}, & \text{ if } i=5;\\
    0, & \text{ otherwise}.
    \end{cases}
\end{equation*}

Which implies that $D_{6,3}$ is not homotopic to a wedge of equi-dimensional spheres. Here $\tilde{H}_i(\K;\mathbb{Z})$ denotes the reduced $i^{\text{th}}$ homology group of simplcial complex $\K$ with integer coefficients.

\subsection{Homotopical depth}
A pure simplicial complex $L$ is called {\itshape homotopically Cohen-Macaulay (CM)} if link of any simplex $\sigma$ in $L$, denoted as lk$_{L}(\sigma)$, is $($dim$($lk$_{L}(\sigma)) -1)$-connected. The {\itshape homotopical depth} of a simplicial complex $\K$ (not necessarily pure) is the largest $k$ such that the $k$-skeleton of $\K$, denoted as $\K^{(k)}$, is pure and homotopically CM. 

It is easy to see that the homotopical depth of a pure simplicial complex $\K$ is at most the shifted connectivity degree of $\K$. The homotopical depth of $M_1(K_n)$ is known to be equal to the shifted connectivity degree of $M_1(K_n)$ which is $\lceil{\frac{n-4}{3}}\rceil$, see \cite[Corollary 11.13]{Jon08}. In \cite[Proposition 12.11]{Jon08}, Jonsson showed that the homotopical depth of $M_2(K_n)$ is at least $\lceil{\frac{3n-7}{4}}\rceil$. In this direction, we strongly believe that the following is true.

\begin{conjecture}\label{conjecture}
The homotopical depth of $M_{n-2}(K_n)$ and $M_{n-1}(K_{n,n})$ is equal to the respective shifted connectivity degree.
\end{conjecture}

The purity of $M_{n-2}(K_n)^{\big{(}\binom{n-1}{2}-1\big{)}}$ and $M_{n-1}(K_{n,n})^{((n-1)^2-1)}$ comes from the following observation.

\begin{proposition}
If $G \in M_{n-2}(K_n)$ and $H \in M_{n-1}(K_{n,n})$ are facets in respective complexes then $|E(G)|\geq \binom{n-1}{2}$ and $|E(H)|\geq (n-1)^2$.
\end{proposition}
\begin{proof}
We prove the result for $H \in M_{n-1}(K_{n,n})$. Proof for $G \in M_{n-2}(K_n)$ will follow using similar arguments. Let $H \in M_{n-1}(K_{n,n})$ and $|E(H)|< (n-1)^2$. Let $A=\{a_i: i \in [n]\}$ and $B=\{b_i: b \in [n]\}$ be the partition of vertices of $K_{n,n}$ as in \Cref{notation:bip}. To show that $H$ is not a facet, we need to find two non-adjacent vertices (one from each partition of $V(K_{n,n})$) with degree less than $n-1$. Suppose that there does not exist such pair, {\itshape i.e.}, $H$ is facet. Let $C \subseteq A$ and $D\subseteq B$ such that $\deg_H(x)<n-1$ if and only if $x\in C \sqcup D$. Since $|E(H)|< (n-1)^2$, $C\neq \emptyset \neq D$. Moreover, there does not exist a non-adjacent desired pair implies that $C\sqcup D$ form a complete bipartite subgraph of $H$. Assuming that $|C|=c$, and $|D|=d$, we count the number of edges of $H$.
\begin{equation*}
\begin{split}
|E(H)| & \geq (n-1)(n-c)+cd+(c-1)(n-d)\\
       & = n^2-n-cn+c+cd-cd+cn-n+d\\
       & = n^2-2n+c+d.
\end{split}
\end{equation*}

Our assumption thus implies that $n^2-2n+c+d < (n-1)^2 = n^2-2n+1$. Thus $c+d<1$ and this contradicts the fact that $C\neq \emptyset \neq D$.
\end{proof}

\section{Acknowledgements}
The author would like to thank an anonymous reader for pointing out the stable homotopy equivalence between $M_{n-1}(K_{n,n})$ and $\mathbb{S}^{(n-1)^2-1}$.

 \bibliographystyle{plain}
 
\end{document}